\documentclass[11pt,psfig]{article}

\usepackage{amssymb}
\usepackage{pict2e}
\usepackage{amsthm}
\usepackage{fleqn}
\usepackage{graphicx}
\usepackage{epstopdf}
\usepackage{indentfirst}
\usepackage{amsfonts}
\usepackage[numbers]{natbib}
\usepackage[colorlinks,citecolor=blue,urlcolor=blue]{hyperref}
\usepackage{amsmath}
\usepackage{amstext}
\usepackage{enumitem}
\newtheorem{theorem}{Theorem}[section]
\newtheorem{lemma}[theorem]{Lemma}
\newtheorem{proposition}[theorem]{Proposition}
\newtheorem{corollary}[theorem]{Corollary}
\newtheorem{remark}[theorem]{Remark}

\newtheorem{assumption}[theorem]{Assumption}
\newtheorem{example}[theorem]{Example}
\newcounter{neweqn}

\newcommand{\beq}[1]{\begin{equation} \refstepcounter{neweqn} \label{#1}}
\newcommand{\eeq}{\end{equation}}
\newcommand{\bed}{\begin{displaymath}}
\newcommand{\eed}{\end{displaymath}}
\newcommand{\bedd}{\bed\begin{array}{l}}
\newcommand{\eedd}{\end{array}\eed}

\newcommand{\al}{\alpha}

\def\({\left(}
\def\){\right)}

\begin{document}

\title{Two Equivalent Families of Linear Fully Coupled Forward Backward Stochastic Differential Equations}

\author{Ruyi Liu\thanks{School of Mathematics and Statistics, University of Sydney, NSW 2006, Australia. Email: ruyi.liu@sydney.edu.au}
\and
Zhen Wu\thanks{School of Mathematics,
Shandong University, Jinan 250100,
People's Republic of China
wuzhen@sdu.edu.cn.}
\and
Detao Zhang\thanks{School of Economics, Shandong University, Jinan 250100, People’s Republic of China, zhangdetao@sdu.edu.cn}
}

\maketitle

\date{}

\begin{abstract}
In this paper, we investigate two families of fully coupled linear Forward-Backward Stochastic Differential Equations (FBSDE). Within these families, one could get the same well-posedness of FBSDEs with totally different structures. The first family of FBSDEs are proved to be equivalent with respect to the Unified Approach. Thus one could get the well-posedness of the whole family if one member exists a unique solution. Another equivalent family of FBSDEs are investigated by introducing a linear transformation method. By reason of the fully coupling structure between the forward and backward equations, it leads to a highly interdependence in solutions. We are able to lower the coupling of FBSDEs, by virtue of the idea of transformation, without losing the well-posedness. Moreover, owing to the non-degeneracy of the transformation matrix, the solution to original FBSDE is totally determined by solutions of FBSDE after transformation. In addition, an example of optimal Linear Quadratic (LQ) problem is presented to illustrate.

\bigskip\noindent
{\bf Key words:} Forward-Backward Stochastic Differential Equations, Unified Approach, Linear Transformation Method, Linear Quadratic Problem

\bigskip\noindent 
{\bf MSC (2010):}  39A50, 60G99, 93E20

\end{abstract}
\section{Introduction}
Suppose that $(\Omega, \mathcal{F}, \mathbb{P})$ be a filtered probability space on which $W =(W_t)_{t \geq 0}$ is defined a standard Brownian motion. We assume $\mathcal{F}=\{\mathcal{F}_t\}_{t\geq 0}$ to be the natural filtration generated by $W_t$ augmented by the null sets of $\mathbb{P}$. 
A general fully coupled FBSDE takes the form:

\beq{Nonl}
\begin{cases}
 dX(t)=b(t,X(t),Y(t),Z(t))dt+\sigma(t,X(t),Y(t),Z(t))dW(t),
 \\
 -dY(t)=f(t,X(t),Y(t),Z(t))dt-Z(t)dW(t),
 \\
 X_0=x \ \ \ Y(T)=h[X(T)],  \ \ \ \ \ \ \ \ \ \ \ \  0\leq t \leq T,
\end{cases}
\eeq
where $b, \sigma, f$ are all $\mathcal{F}$-progressively measurable processes defined on appropriate spaces with  $W(t)$ being a standard Brownian motion and $h(\cdot)$ is $\mathcal{F}_T$-measurable function. 

For technical clarity, we employ the following standard assumptions throughout the paper.  
\begin{assumption}
{\bf (H1)} The following condition holds for coefficient $b,\sigma, f, h$ on $[0,T]:$    
\bed
E\lbrace(\int_0^T[|b(t,0,0,0)|+|f(t,0,0,0)|]dt)^2+\int_0^T|\sigma(t,0,0,0)|^2 dt+|h(0)|^2 \rbrace < \infty
\eed
~\\
{\bf (H2) } The coefficients $b,\sigma, f$ and $ h$ are uniformly Lipschitz continuous with respect to $(x,y,z)$  and $x$ respectively with a Lipschitz constant $L> 0$.
\end{assumption}

To get the well-posdeness of nonlinear FBSDE~(\ref{Nonl}), the Method of Contract mapping was firstly introduced by Antonelli \cite{Antonelli1993} and later detailed by Pardoux and Tang \cite{Pardoux1999} to solve FBSDE with relative small time duration $T$. Afterwards, in \cite{Ma1994}, Ma, Protter and Yong  introduced the Four Step Scheme to handle the case of arbitrary duration $T$, which requires regularity assumption on coefficients in Markovian structure.

 To investigate FBSDE~(\ref{Nonl}) with non-Markovian coefficients, the Method of Continuation  was introduced (see Hu and Peng \cite{Hu-Peng1995}, Peng and Wu \cite{Peng-Wu1999}, Yong \cite{Yong1997Finding}) into the literature. But as a trade off, it required a Monotonicity conditions as the follows: 
\begin{assumption}
 {\bf (H3)}
Assume that $x,y,z$ are same dimension. Denoting $\theta_i=(x_i, y_i, z_i), i=1,2$ and $A(t,\theta_i)=\left(                 
  \begin{array}{c}   
    -f \\
    b \\
    \sigma \\  
  \end{array}
\right)(t,\theta_i)$, there exist some constants $\beta_1, \beta_2, \mu >0$, for any $\hat{x}=x_1-x_2, \ \hat{y}=y_1-y_2, \ \hat{z}=z_1-z_2$, such that  
\beq{Mc}
\begin{cases}
\langle A(t,\theta_1)-A(t,\theta_2), \theta_1-\theta_2 \rangle \ \leq  -\beta_1|\hat{x}|^2- \beta_2(|\hat{y}|^2+|\hat{z}^2|) ,
\\~\\
\langle h(x_1)-h(x_2), x_1-x_2 \rangle \geq \mu |\hat{x}|^2.
\end{cases}
\eeq
\end{assumption}
Recently, a Unified Approach was introduced by Ma, Wu, Zhang and Zhang \cite{Ma2015On} to investigate nonlinear FBSDE. This method constructed a unified scheme which aims to summarize all existing methodologies in the literature, and provided a series of sufficient and necessary conditions to get the well-posedness of (\ref{Nonl}).

The core to the Unified Approach, analogue to the Four Step Scheme, is to find a decoupling field $u(t,\cdot)$ such that $Y_t=u(t,X_t)$ on $[0,T]$. It will ultimately leads to the well-posedness of (\ref{Nonl}) that the decoupling field $u(t,\cdot)$ is uniformly Lipschitz in its spatial variable. And proving $u$ being uniformly Lipschitz continuous amounts to finding solutions to the following "variational FBSDE":
\beq{VF}
\begin{cases}
\nabla X(t)= & 1+\int_{0}^{t}\left[b_{1} \nabla X(s)+b_{2} \nabla Y(s)+b_{3} \nabla Z(s)\right] d s +\int_{0}^{t}\left[\sigma_{1} \nabla X(s)+\sigma_{2} \nabla Y(s)+\sigma_{3} \nabla Z(s)\right] d W(s) , \\~\\
\nabla Y(t)= & h \nabla X(T)+\int_{t}^{T}\left[f_{1} \nabla X(s)+f_{2} \nabla Y(s)+f_{3} \nabla Z(s)\right] d s -\int_{t}^{T} \nabla Z(s) d W(s), \ \ \ \ \ t \in[0, T],
\end{cases} 
\eeq
where $\nabla \Theta \triangleq \frac{\Theta^{1}-\Theta^{2}}{x_{1}-x_{2}} $ denotes the derivative of $\Theta=X, Y, Z$ with respect to the initial value $x$. Note that (\ref{VF}) is a linear fully coupled FBSDE where the coefficients $b_i, \sigma_i, f_i, i=1,2 ,3$ are bounded in consequence of the Lipschitz condition (H2). 

A simple example, also the motivation for this paper, often appeared in the optimal investment problem and stochastic control problem is of the following form:
\beq{Ex3}
\begin{cases}
 X(t)=x+\int_0^t [a(s)X(s)+b(s)Y(s)+c(s)Z(s)]ds+\int_o^t[d(s)X(s)+e(s)Y(s)+f(s)Z(s)]dW(s),
 \\
 \\
 Y(t)=hX(T)+\int_t^T [m(s)X(s)+p(s)Y(s)+q(s)Z(s)ds]-\int_t^T Z(s)dW(s),  \ \ \ \ \ \ \ \ \ \ \ \  0\leq t \leq T,
\end{cases}
\eeq
where $a(\cdot), b(\cdot), c(\cdot), d(\cdot), e(\cdot), f(\cdot), m(\cdot), p(\cdot), q(\cdot)$ are bounded processes and $h$ is an $\mathcal{F}_{T}$ random variable.
({\ref{Ex3}}) is often derived from the Pontrygin's maximum principle when seeking the closed-loop optimal control of linear quadratic (LQ) optimal control problem. This kind of linear FBSDE are widely applied in many areas, such as ordinary differential equations \cite{2019Ahmad}  \cite{2021Verma}, stochastic control \cite{Yong1999Stochastic} and mathematical finance \cite{Wu2014SPA} \cite{Yu2013Equiva}. However, the well-posedness of (\ref{Ex3}) is not covered by any existing methods despite it is linear, homogeneous and bounded. As we will see in Section 5, the solvability of (\ref{Ex3}) will be a straight consequence of our results.

In this paper, we aim to get the well-posedness of this kind of Linear FBSDE. In Section 2, we formulate the linear fully coupled FBSDE and introduce the monotonicity conditions for it. In Section 3, we discuss a family of linear FBSDE which are proved to be equivalent with respect to the Unified Approach. In Section 4, we introduce a linear transformation method to study some FBSDE which could not be proved to be well posed by the existing methods. In Section 5, we employ the linear transformation to deal with the LQ stochastic control problem.

\section{Notations and Problem Formulation}
First, we introduce the following spaces:
$$L_{\mathcal{F}_{T}}^{\infty}(\Omega ; \mathbb{H})=\left\{\xi: \Omega \rightarrow \mathbb{H} \mid \xi \text { is } \mathcal{F}_{T} \text { -measurable bounded variable} \right\};$$
$$L_{\mathcal{F}}^{\infty}([0, T] ; \mathbb{H})=\{\varphi:[0, T] \times \Omega \rightarrow \mathbb{H} \mid \varphi(\cdot) \text { is } \mathbb{F} \text { -adapted, bounded processes}\}.$$

For the technical clarity, we only consider $\mathbb{H}=\mathcal{R}^1$ for simplification in what follows and multi-dimensional cases can be dealt with similarly. Let $T>0$ be a fixed time horizon and we consider the following linear fully coupled FBSDE:
\beq{LF}
\begin{cases}
 dX(t)=[b_1(t)X(t)+b_2(t)Y(t)+b_3(t)Z(t)]dt+[\sigma_1(t)X(t)+\sigma_2(t)Y(t)+\sigma_3(t)Z(t)]dW(t),
 \\
 -dY(t)=[f_1(t)X(t)+f_2(t)Y(t)+f_3(t)Z(t)]dt-Z(t)dW(t),
 \\
 X_0=x \ \ \ Y(T)=hX(T),  \ \ \ \ \ \ \ \ \ \ \ \  0\leq t \leq T,
\end{cases}
\eeq
where $b_i(\cdot), f_i(\cdot), \sigma_i(\cdot)\in L_{\mathcal{F}}^{\infty}([0, T] ; \mathbb{H}), i=1,2,3$ and $h \in L_{\mathcal{F}_{T}}^{\infty}(\Omega ; \mathbb{H})$ . 
\begin{remark}
{\rm
It is a straight consequence that $b_i, \sigma_i, f_i, i=1, 2, 3$ and $h$ are bounded since $b(t,\cdot,\cdot,\cdot)$,\ $\sigma(t,\cdot,\cdot,\cdot), \ f(t,\cdot,\cdot,\cdot), h(\cdot)$ in (\ref{Nonl}) hold for the Lipschitz conditions (H2) .}
\end{remark}
Here and after, we sometimes suppress t also for the processes $b_i, \sigma_i, f_i, i=1, 2, 3$, for the simplicity of
notations.
Motivated by Assumption 1.2, we have the following monotonicity conditions for linear FBSDE~(\ref{LF}):
\begin{lemma}(Monotonicity Conditions)

\emph{(i) For $\forall \ x, y, z \in R$ and fixed }$t$,
\bed
\left(                 
  \begin{array}{ccc}   
    x & y & z \\  
  \end{array}
\right)
\left(                 
  \begin{array}{ccc}   
    -f_1 & -f_2 & -f_3\\
    b_1 & b_2 & b_3  \\
    \sigma_1 & \sigma_2 & \sigma_3  \\  
  \end{array}
\right)
\left(                 
  \begin{array}{c}   
    x \\
    y \\
    z \\  
  \end{array}
\right)
\leq -\beta_1{|x|}^2-\beta_2({|y|}^2+{|z|}^2),
\eed
\emph{where $\beta_1$ and $\beta_2$ are nonnegative constants. When $\beta_1 > 0, h > 0$, then $\beta_2 \geq 0$; When $\beta_2 > 0$, then $\beta_1 \geq 0, h \geq 0$.}\\

\emph{(ii) For $\forall \ x, y, z \in R$ and fixed }$t$,
\bed
\left(                 
  \begin{array}{ccc}   
     x & y & z \\  
  \end{array}
\right)
\left(                 
  \begin{array}{ccc}   
     -f_1 & -f_2 & -f_3\\
    b_1 & b_2 & b_3  \\
    \sigma_1 & \sigma_2 & \sigma_3  \\ 
  \end{array}
\right)
\left(                 
  \begin{array}{c}   
    x \\
    y \\
    z \\  
  \end{array}
\right)
\geq \beta_1{|x|}^2+\beta_2({|y|}^2+{|z|}^2),  
\eed
\emph{where $\beta_1$ and $\beta_2$ are nonnegative constants. When $\beta_1 > 0, h < 0$, then $\beta_2 \geq 0$; When $\beta_2 > 0$ , then $\beta_1 \geq 0, h \leq 0$.}
\end{lemma}

\section{Equivalent Coefficient Matrix for Unified Approach}
 
The Unified Approach is one of the principal methods to solve the well-posedness of (\ref{LF}). By virtue of the variational FBSDE, Ma et al. \cite{Ma2015On} derived the dominating ODE of (\ref{LF}) which takes the following form:
\beq{Dode}
\mathbf{y}_{t}=h+\int_{t}^{T} F\left(s,\mathbf{y}_{s}\right) ds,  \ \ \ \  \ 0 \leq s \leq T,
\eeq
where
\beq{F}
F(s,y) \triangleq f_{1}(s)+f_{2}(s) y+y\left[b_{1}(s)+b_{2}(s) y\right]+\frac{\left(f_{3}(s)+b_{3}(s) y\right) y\left(\sigma_{1}(s)+\sigma_{2}(s) y\right)}{1-\sigma_{3}(s) y} .
\eeq
And the well-posedness of (\ref{LF}) is a straightforward consequence of the boundary solutions to (\ref{F}) according to the following lemma.
\begin{lemma}(Ma, Wu, Zhang, Zhang \cite{Ma2015On})
Linear FBSDE~(\ref{LF}) is well posed if and only if the dominating ODE~(\ref{Dode}) exists boundary upper/lower solutions $\overline{\mathbf{y}}_{t}, \underline{\mathbf{y}}_{t}$ on $[0,T]$:
\bed
\overline{\mathbf{y}}_{t}=\bar{h}+\int_{t}^{T} \bar{F}\left(s, \overline{\mathbf{y}}_{s}\right) d s, \quad \underline{\mathbf{y}}_{t}=\underline{h}+\int_{t}^{T} \underline{F}\left(s, \underline{\mathbf{y}}_{s}\right) d s
\eed
where  $\bar{h}, \underline{h}$ denote the upper/lower bound of $h$ and
\bed
\bar{F}(t, y):=\operatorname{esssup} F(t, y), \quad \underline{F}(t, y):=\operatorname{essinf} F(t, y) .
\eed
\end{lemma}
It is noted that some coefficients in (\ref{F}) are symmetric which indicates some different FBSDE correspond to a same dominating ODE (\ref{Dode}).

Denoting $F_s(y)=F(s,y)$, we have
\beq{Fs}
F_s(y)=\frac{\left(b_{3} \sigma_{2}-b_{2} \sigma_{3}\right) y^{3}+\left(b_{2}+\sigma_{2} f_{3}-\sigma_{3} f_{2}+\sigma_{1} b_{3}-\sigma_{3} b_{1}\right) y^{2}+\left(f_{2}+b_{1}+f_{3} \sigma_{1}-f_{1} \sigma_{3}\right) y+f_{1}}{1-\sigma_3 y}
\eeq

Note that 
\bed
b_{3} \sigma_{2}-b_{2} \sigma_{3}=-\left|                
  \begin{array}{cc}   
    b_2 & b_3 \\
    \sigma_2 & \sigma_3 \\  
  \end{array}
\right|,
\ \ \ 
\sigma_2f_3-\sigma_3f_2=\left|                
  \begin{array}{cc}   
    -f_2 & -f_3 \\
    \sigma_2 & \sigma_3 \\  
  \end{array}
\right|,
\eed
\bed
\sigma_1b_3-\sigma_3 b_1=-\left|                
  \begin{array}{cc}   
    b_1 & b_3 \\
    \sigma_1 & \sigma_3 \\  
  \end{array}
\right|,
\ \ \ 
f_3\sigma_1-f_1\sigma_3=\left|                
  \begin{array}{cc}   
    -f_1 & -f_3 \\
    \sigma_1 & \sigma_3 \\  
  \end{array}
\right|,
\eed
where $|\cdot|$ denotes the determinant of matrix.

Rewrite (\ref{Fs}) in terms of these determinants, and we have
\beq{Fm}
F_s(y)=\frac{-\left|                
  \begin{array}{cc}   
    b_2 & b_3 \\
    \sigma_2 & \sigma_3 \\  
  \end{array}
\right| y^{3}+\left(b_{2}+
\left|                
  \begin{array}{cc}   
    -f_2 & -f_3 \\
    \sigma_2 & \sigma_3 \\  
  \end{array}
\right|
-\left|                
  \begin{array}{cc}   
    b_1 & b_3 \\
    \sigma_1 & \sigma_3 \\  
  \end{array}
\right|\right) y^{2}
+
\left(f_{2}+b_{1}+
\left|                
  \begin{array}{cc}   
    -f_1 & -f_3 \\
    \sigma_1 & \sigma_3 \\  
  \end{array}
\right|
\right) y+f_{1}}{1-\sigma_3 y}
\eeq

\begin{proposition}
Let $\small A= \left(                
  \begin{array}{ccc}   
    -f_1 & -f_2 & -f_3 \\
        b_1 & b_2 & b_3 \\
    \sigma_1 & \sigma_2 & \sigma_3 \\  
  \end{array}
\right)$  denote the coefficient matrix of linear FBSDE ~(\ref{LF}), then for $ \forall \ p \in R$, 
$B=\small \left(                
  \begin{array}{ccc}   
    -f_1 & -f_2 & -f_3+p \\
        b_1+\sigma_1 p & b_2 +\sigma_2p & b_3+\sigma_3 p \\
    \sigma_1 & \sigma_2 & \sigma_3 \\  
  \end{array}
\right)$ is an equivalent coefficient matrix to $A$ with respect to the dominating function ~(\ref{F}).
\end{proposition}
\begin{proof}
Note that 

\bed
-\left|                
  \begin{array}{cc}   
    b_2+\sigma_2p & b_3+\sigma_3 p \\
    \sigma_2 & \sigma_3 \\  
  \end{array}
\right|=b_{3} \sigma_{2}-b_{2} \sigma_{3},
\ \ \ 
\left|                
  \begin{array}{cc}   
    -f_2 & -f_3+p \\
    \sigma_2 & \sigma_3 \\  
  \end{array}
\right|=\sigma_2f_3-\sigma_3f_2- \sigma_2 p,
\eed
\bed
-\left|                
  \begin{array}{cc}   
    b_1+\sigma_1 p & b_3 +\sigma_3 p \\
    \sigma_1 & \sigma_3 \\  
  \end{array}
\right|=\sigma_1b_3-\sigma_3 b_1,
\ \ \ 
\left|                
  \begin{array}{cc}   
    -f_1 & -f_3+p \\
    \sigma_1 & \sigma_3 \\  
  \end{array}
\right|=f_3\sigma_1-f_1\sigma_3-\sigma_1 p,
\eed
Substituting all the equations above into (\ref{Fm}), we have 
\bed
\begin{split}
\tilde{F}_s(y)=\frac{-\left|                
  \begin{array}{cc}   
    b_2+\sigma_2p & b_3+\sigma_3 p \\
    \sigma_2 & \sigma_3 \\  
  \end{array}
\right| y^{3}+\left(b_{2}+\sigma_2 p+
\left|                
  \begin{array}{cc}   
    -f_2 & -f_3 +p\\
    \sigma_2 & \sigma_3 \\  
  \end{array}
\right|
-\left|                
  \begin{array}{cc}   
    b_1+\sigma_1 p & b_3+\sigma_3 p \\
    \sigma_1 & \sigma_3 \\  
  \end{array}
\right|\right) y^{2}}{1-\sigma_3 y}
\\
+\frac{
\left(f_{2}+b_{1}+\sigma_1 p+
\left|                
  \begin{array}{cc}   
    -f_1 & -f_3 +p\\
    \sigma_1 & \sigma_3 \\  
  \end{array}
\right|
\right) y+f_{1}}{1-\sigma_3 y}=F_s(y)
\end{split}
\eed
This implies the equivalence between $A$ and $B$ in ~(\ref{F}).
\end{proof}

\begin{remark}
{\rm For multi-dimensional cases, we assume $b_i, f_i, \sigma_i, p \in \mathcal{R}^{m \times m}$. It is not trivial to get
$$-\left|                
  \begin{array}{cc}   
    b_2+\sigma_2p & b_3+\sigma_3 p \\
    \sigma_2 & \sigma_3 \\  
  \end{array}
\right|=b_{3} \sigma_{2}-b_{2} \sigma_{3}.$$
Then it suffices to assume that $b_i, f_i, \sigma_i, p$ are all symmetric.}
\end{remark}

Similarly, we could get another family of equivalent coefficient matrix to $A$.
\begin{corollary}
Let $\small A= \left(                
  \begin{array}{ccc}   
    -f_1 & -f_2 & -f_3 \\
        b_1 & b_2 & b_3 \\
    \sigma_1 & \sigma_2 & \sigma_3 \\  
  \end{array}
\right)$  denotes the coefficient matrix of linear FBSDE ~(\ref{LF}), then for $ \forall \ q \in R$, 
$C=\small \left(               
  \begin{array}{ccc}   
    -f_1 & -f_2-f_3q & -f_3 \\
        b_1 & b_2 +b_3q & b_3 \\
    \sigma_1-q & \sigma_2+\sigma_3q & \sigma_3 \\  
  \end{array}
\right)$ is an equivalent coefficient matrix to $A$ with respect to the dominating function ~(\ref{F}).
\end{corollary}

\begin{remark}
{\rm
Note that $b_1$ and $f_2$ are symmetric in (\ref{Fs}), then it is also an equivalent coefficient matrix for $B=\footnotesize \left(                
  \begin{array}{ccc}   
    -f_1 & -f_2-\sigma_1 p & -f_3+p \\
        b_1 & b_2 +\sigma_2p & b_3+\sigma_3 p \\
    \sigma_1 & \sigma_2 & \sigma_3 \\  
  \end{array}
\right)$ in Proposition 3.2 and $C=\footnotesize \left(                
  \begin{array}{ccc}   
    -f_1 & -f_2 & -f_3 \\
        b_1+f_3q & b_2 +b_3q & b_3 \\
    \sigma_1-q & \sigma_2+\sigma_3q & \sigma_3 \\  
  \end{array}
\right)$ in Corollary 3.4, respectively.}
\end{remark}

\begin{remark}
{\rm
In Proposition 3.2, for any $\lambda \in R$, the equivalent matrix $B=\footnotesize \left(               
  \begin{array}{ccc}   
    -f_1 & -f_2 & -f_3+\lambda \\
        b_1+\sigma_1 \lambda & b_2 +\sigma_2 \lambda & b_3+\sigma_3\lambda \\
    \sigma_1 & \sigma_2 & \sigma_3 \\  
  \end{array}
\right)$ corresponds to the following FBSDE:
\beq{LF1}
\begin{cases}
 dX(t)=[b_1X(t)+b_2Y(t)+b_3Z(t)]dt+[\sigma_1 X(t)+\sigma_2 Y(t)+\sigma_3 Z(t)]d\tilde{W}(t),
 \\
 -dY(t)=[f_1X(t)+f_2Y(t)+f_3Z(t)]dt-Z(t)d\tilde{W}(t),
 \\
 X_0=x \ \ \ Y(T)=hX(T),  \ \ \ \ \ \ \ \ \ \ \ \  0\leq t \leq T,
\end{cases}
\eeq
where $d\tilde{W}(t)= dW(t)+\lambda dt.$

By virtue of Girsanov Transformation,  for $\forall t\geq 0$, we have an equivalent probability $\mathbb{Q}$ with respect to $\mathbb{P}$ such that
$$\frac{d\mathbb{Q}\ |\mathcal{F}_t}{d\mathbb{P}\ |\mathcal{F}_t}=\exp{\{-\frac{1}{2}\lambda^2t-\lambda W(t) \}},$$
and $\tilde{W}(t)$ is a standard Brownian motion under probability $\mathbb{Q}$.}
\end{remark}

It is noted that, in remark 3.6, dominating function~(\ref{F}) has the same structure under equivalent probability measure $\mathbb{P}$ and $\mathbb{Q}$. As a consequence, the well-posedness of FBSDE under $\mathbb{Q}$ is equivalent to the FBSDE (\ref{LF}) under $\mathbb{P}$.
Next we present the main result for this section.
\begin{theorem}
For any R-valued bounded process $  p, q :[0,T]\mapsto R $ , the well-posedness of the following FBSDE is equivalent to FBSDE~(\ref{LF}):

\beq{LF2}
\begin{cases}
 dX(t)=[(b_1+\sigma_1) X(t)+(b_2+\sigma_2 p+b_3q+\sigma_3 pq)Y(t)+(b_3+\sigma_3 p)Z(t)]dt
\\ 
\qquad\qquad\qquad\qquad\qquad\qquad\qquad\qquad\qquad+[(\sigma_1-q) X(t) +(\sigma_2+\sigma_3 p) Y(t)+\sigma_3 Z(t)] dW(t),
 \\
 -dY(t)=[f_1X(t)+(f_2+f_3q-pq)Y(t)+(f_3-p)Z(t)]dt-Z(t)dW(t),
 \\
 X_0=x \ \ \ Y(T)=hX(T),  \ \ \ \ \ \ \ \ \ \ \ \  0\leq t \leq T.
\end{cases}
\eeq
\end{theorem}
\begin{proof}
Note that the coefficient matrix of (\ref{LF2}) is as the follows:

$$D= \left(               
  \begin{array}{ccc}   
    -f_1 & -f_2+(-f_3+p)q & -f_3+p\\
        b_1+\sigma_1 p & b_2 +\sigma_2 p+(b_3+\sigma_3p)q & b_3+\sigma_3p \\
    \sigma_1-q & \sigma_2+\sigma_3q & \sigma_3 \\  
  \end{array}
\right).$$

According to Corollary 3.4, $D$ is an equivalent coefficient matrix to $B=\small \left(                
  \begin{array}{ccc}   
    -f_1 & -f_2 & -f_3+p \\
        b_1+\sigma_1 p & b_2 +\sigma_2p & b_3+\sigma_3 p \\
    \sigma_1 & \sigma_2 & \sigma_3 \\  
  \end{array}
\right)$.
It is a straight consequence, by Proposition 3.2, that $B$ is equivalent to 
$\small A= \left(                
  \begin{array}{ccc}   
    -f_1 & -f_2 & -f_3 \\
        b_1 & b_2 & b_3 \\
    \sigma_1 & \sigma_2 & \sigma_3 \\  
  \end{array}
\right)$.
It is easy to get that FBSDE~(\ref{LF2}) and (\ref{LF}) has the same dominating ODE~(\ref{Dode}). This completes the proof by Lemma 3.1.
\end{proof}
We can also formulate other equivalent matrix with respect to $A$, by applying  Proposition 3.2, Corollary 3.4 and Remark 3.5 sequentially, and the proof would be similar.

 Next we employ the results in this section to investigate FBSDE~(\ref{LF}) with all the coefficients constant. In this case, $\overline{h}=\underline{h}=h\in R $, and 
\bed
\begin{split}
F(y)&=\overline{F}(t,y)=\underline{F}(t,y)
 \\
 &=\small \frac{\left(b_{3} \sigma_{2}-b_{2} \sigma_{3}\right) y^{3}+\left(b_{2}+\sigma_{2} f_{3}-\sigma_{3} f_{2}+\sigma_{1} b_{3}-\sigma_{3} b_{1}\right) y^{2}+\left(f_{2}+b_{1}+f_{3} \sigma_{1}-f_{1} \sigma_{3}\right) y+f_{1}}{1-\sigma_3 y},
 \end{split}
 \eed
 where $b_i, \sigma_i, f_i, i=1,2,3$ are all constants.
 
 By virtue of the Unified Approach,  Ma et al. \cite{Ma2015On} presented a necessary and sufficient condition for the existence and uniqueness of solutions to such cases:
 \begin{lemma} If the coefficients $b_i, \sigma_i, f_i, i=1,2,3$ and $h$ are all constants, then linear FBSDE~(\ref{LF}) exists a unique solution for arbitrary $T>0$ and terminal condition $h \ (h\neq \frac{1}{\sigma_3})$ if and only if one of the following cases holds:
 \begin{itemize}
\item[(i)]  \ \ $h<\frac{1}{\sigma_3}, \ \ F(h)\leq 0, $ and either\ \ $ F(y) $ has a zero point in $(-\infty,h]$ or $ b_3\sigma_2-b_2\sigma_3=0$;
 
\item[(ii)]  \ \ $h>\frac{1}{\sigma_3}, \ \ F(h)\geq 0, $ and either\ \ $ F(y) $ has a zero point in $[h,\infty)$ or $ b_3\sigma_2-b_2\sigma_3=0$;

\item[(iii)]  \ \ $h<\frac{1}{\sigma_3}, \ \ F(h)\geq 0,$ \ \  and F has a zero point in $[h,\frac{1}{\sigma_3}]$;

\item[(iv)]  \  \ $h>\frac{1}{\sigma_3}, \ \ F(h)\leq 0,$ \ \  and F has a zero point in $[\frac{1}{\sigma_3},h]$;
\end{itemize}
 
 \end{lemma}
 
 



Note that it is difficult to get all zero points of $F(y)$. To make full use of Lemma 3.8, we need to simplify the criterion which is easy to check.  Denoting
\bed
 \mathcal{L}(y)=(1-\sigma_3 y)F(y)=\left(b_{3} \sigma_{2}-b_{2} \sigma_{3}\right) y^{3}+\left(b_{2}+\sigma_{2} f_{3}-\sigma_{3} f_{2}+\sigma_{1} b_{3}-\sigma_{3} b_{1}\right) y^{2}+\left(f_{2}+b_{1}+f_{3} \sigma_{1}-f_{1} \sigma_{3}\right) y+f_{1},
 \eed 
   we present the following sufficient condition:
 \begin{theorem}
 Assume all coefficients $b_i, \sigma_i, f_i, h$ are constants in the linear FBSDE~(\ref{LF}). Then it is well-posed for arbitrary $T>0$ if one of the following cases hold true:
 \begin{itemize}
 \item[(i)] $1-\sigma_3 h > 0, \   b_3\sigma_2-b_2\sigma_3 \leq 0$ and \ $\mathcal{L}(h)\cdot \sigma_3 \leq 0$; 
 \item[(ii)] $1-\sigma_3 h < 0, \   b_3\sigma_2-b_2\sigma_3 \geq 0$ and \ $\mathcal{L}(h)\cdot \sigma_3 \leq 0$; 
 \item[(iii)] $\sigma_3>0, \ \mathcal{L}(\frac{1}{\sigma_3}) \leq 0$ and $\mathcal{L}(h) \geq 0;$
  \item[(iv)] $\sigma_3<0, \ \mathcal{L}(\frac{1}{\sigma_3}) \geq 0$ and $\mathcal{L}(h) \leq 0;$
 \end{itemize}
 \end{theorem}
 \begin{proof}
 (i) Note that the proof is trivial for $b_3\sigma_2-b_2\sigma_3 = 0$. 
 
 For cases of $b_3\sigma_2-b_2\sigma_3<0$, if $\sigma_3>0$, we have 
$$h<\frac{1}{\sigma_3} \ \ \text{and}\ \  \mathcal{L}(h)\leq 0.   $$
It indicates that 
 $$F(h)=\frac{\mathcal{L}(h)}{1-\sigma_3h}\leq 0.$$
 
Note that $b_3\sigma_2-b_2\sigma_3 < 0$ is the coefficient of $y^3$ in $\mathcal{L}(y).$ It is easy to get a $\mathcal{L}$ has a zero point in $(-\infty,h]$ which coincides with case (i) of Lemma 3.8.

 If $\sigma_3<0$, we have 
 $$h>\frac{1}{\sigma_3} ,\  \mathcal{L}(h)\geq 0  \ \text{and}\ \ F(h)=\frac{\mathcal{L}(h)}{1-\sigma_3h}\geq 0.$$ And $b_3\sigma_2-b_2\sigma_3 < 0$ leads to that $L$ has a zero point in $[h,+\infty)$ which corresponds to case (ii) of Lemma 3.8.
 
 (ii) can be proved in a similar way.
 
 (iii) When $\sigma_3>0$, we first assume $h<\frac{1}{\sigma_3}$. We can get $F(h)\geq 0$ owing to $\mathcal{L}(h) \geq 0$.
 
  In addition, it is noted that $\mathcal{L}(\frac{1}{\sigma_3}) < 0$. Owing to $\mathcal{L}(\cdot)$ being a continuous function, then there exists a constant $\eta< \frac{1}{\sigma_3}$ such that $$\mathcal{L}(\eta)\leq 0 .$$
 It follows that $\mathcal{L}$ has a zero point in $[h,\eta]$ which coincides with case (iii) of Lemma 3.8.
 Similarly, if $h>\frac{1}{\sigma_3}$, this case coincides with case (iv) of Lemma 3.8.
 
  (iv) can be proved similarly which completes the proof.
 \end{proof}
 
  \begin{remark}
{\rm  
However, the monotonicity condition (Lemma 2.2) is a sufficient condition. Actually, Liu and Wu \cite{LiuWu} proved that, for (\ref{LF}) with constant coefficients, the monotonicity condition is indeed a special case of the unified approach. }
  \end{remark}
  
  Then, by employing the results in this section, we investigate the difference between the monotonicity condition (Lemma 2.2) and the unified approach. For cases satisfying the unified approach while the monotonicity condition do not hold, we apply our results to derive some feasible values of $p$ such that the equivalent coefficient matrix also holds for the monotonicity conditions. 

  Note that the equivalent coefficient matrix $B=\small \left(                
  \begin{array}{ccc}   
    -f_1 & -f_2 & -f_3+p \\
        b_1+\sigma_1 p & b_2 +\sigma_2p & b_3+\sigma_3 p \\
    \sigma_1 & \sigma_2 & \sigma_3 \\  
  \end{array}
\right)$ and $C=\small \left(               
  \begin{array}{ccc}   
    -f_1 & -f_2-f_3q & -f_3 \\
        b_1 & b_2 +b_3q & b_3 \\
    \sigma_1-q & \sigma_2+\sigma_3q & \sigma_3 \\  
  \end{array}
\right)$ could be transformed into the symmetric structure for Lemma 2.2:
\bed
\small \hat{B}= \left(                
  \begin{array}{ccc}   
    -f_1 & (b_1-f_2+\sigma_1 p)/2 & (\sigma_1-f_3+p)/2 \\[0.05in]
        (b_1-f_2+\sigma_1 p)/2 & b_2 +\sigma_2p & (\sigma_2+b_3+\sigma_3 p)/2 \\[0.05in]
    (\sigma_1-f_3+p)/2 & (\sigma_2+b_3+\sigma_3 p)/2 & \sigma_3 \\  
  \end{array}
\right) 
\eed
and 
\bed
\qquad\qquad\qquad \small \hat{C}=\left(               
  \begin{array}{ccc}   
    -f_1 & ( b_1-f_2-f_3q )/2 & (\sigma_1-f_3-q )/2 \\[0.05in]
       ( b_1-f_2-f_3q )/2 & b_2 +b_3q & (b_3+\sigma_2+\sigma_3q)/2 \\[0.05in]
    (\sigma_1-f_3-q )/2& (b_3+\sigma_2+\sigma_3q)/2 & \sigma_3 \\  
  \end{array}
\right), \ \text{respectively.}
\eed

Here we present a theorem to determine the value of $p$.
\begin{theorem}
The equivalent coefficient matrix $B$ holds for monotonicity conditions (Lemma 2.2) if $p$ satisfies one of the following criterion
\begin{itemize}
\item[(i)]  $h<0$, $f_1<0$ and
\beq{cri1}
\begin{cases}
\small\left|                
  \begin{array}{cc}   
    -f_1 & (b_1-f_2+\sigma_1 p)/2  \\[0.05in]
        (b_1-f_2+\sigma_1 p)/2 & b_2 +\sigma_2p  \\
  \end{array}
\right| > 0 ,
\\[0.25in]
\small \left|                
  \begin{array}{ccc}   
    -f_1 & (b_1-f_2+\sigma_1 p)/2 & (\sigma_1-f_3+p)/2 \\[0.05in]
        (b_1-f_2+\sigma_1 p)/2 & b_2 +\sigma_2p & (\sigma_2+b_3+\sigma_3 p)/2 \\[0.05in]
    (\sigma_1-f_3+p)/2 & (\sigma_2+b_3+\sigma_3 p)/2 & \sigma_3 \\  
  \end{array}
\right| > 0 ;
\end{cases}
\eeq
\item[(ii)]  $h>0$, $f_1>0$ and
\beq{cri2}
\begin{cases}
\small\left|                
  \begin{array}{cc}   
    -f_1 & (b_1-f_2+\sigma_1 p)/2  \\[0.05in]
        (b_1-f_2+\sigma_1 p)/2 & b_2 +\sigma_2p  \\
  \end{array}
\right| > 0 ,
\\[0.25in]
\small \left|                
  \begin{array}{ccc}   
    -f_1 & (b_1-f_2+\sigma_1 p)/2 & (\sigma_1-f_3+p)/2 \\[0.05in]
        (b_1-f_2+\sigma_1 p)/2 & b_2 +\sigma_2p & (\sigma_2+b_3+\sigma_3 p)/2 \\[0.05in]
    (\sigma_1-f_3+p)/2 & (\sigma_2+b_3+\sigma_3 p)/2 & \sigma_3 \\  
  \end{array}
\right| < 0 ;
\end{cases}
\eeq

\end{itemize}

\end{theorem}
  \begin{proof}
(i) Note that $\hat{B}$ is positive definite if ~(\ref{cri1}) and $f_1<0$ hold. Then for $\forall (x,y,z) \in \mathbb{R}^3$, there exist constants $\al_1, \al_2, \al_3 > 0 $ such that 
\bed
\begin{split}
&\left(                 
  \begin{array}{ccc}   
    x & y & z \\  
  \end{array}
\right)
\left(                
  \begin{array}{ccc}   
    -f_1 & (b_1-f_2+\sigma_1 p)/2 & (\sigma_1-f_3+p)/2 \\[0.05in]
        (b_1-f_2+\sigma_1 p)/2 & b_2 +\sigma_2p & (\sigma_2+b_3+\sigma_3 p)/2 \\[0.05in]
    (\sigma_1-f_3+p)/2 & (\sigma_2+b_3+\sigma_3 p)/2 & \sigma_3 \\  
  \end{array}
\right) 
\left(                 
  \begin{array}{c}   
    x \\
    y \\
    z \\  
  \end{array}
\right)
\\[0.1in]
&=\left(                 
  \begin{array}{ccc}   
    x & y & z \\  
  \end{array}
\right)
\left(                
  \begin{array}{ccc}   
    -f_1 & -f_2 & -f_3+p \\
        b_1+\sigma_1 p & b_2 +\sigma_2p & b_3+\sigma_3 p \\
    \sigma_1 & \sigma_2 & \sigma_3 \\  
  \end{array}
\right)
\left(                 
  \begin{array}{c}   
    x \\
    y \\
    z \\  
  \end{array}
\right)
\\[0.1in]
&= \alpha_1{|x|}^2+\alpha_2{|y|}^2+\alpha_3{|z|}^2, 
\end{split}
\eed
which completes the proof according to the case (ii) of Lemma 2.2.

(ii) can be proved in a similar way.
  \end{proof}
  Similarly, we can also determine the value of $q$ by virtue of $\hat{C}$.

  \begin{corollary}
  The equivalent coefficient matrix $C$ holds for monotonicity conditions (Lemma 2.2) if $q$ satisfies one of the following criterion
\begin{itemize}
\item[(i)]  $h<0$, $f_1<0$ and
\beq{cri3}
\begin{cases}
\small\left|                
  \begin{array}{cc}   
     -f_1 & ( b_1-f_2-f_3q )/2  \\[0.05in]
        ( b_1-f_2-f_3q )/2 & b_2 +b_3q  \\
  \end{array}
\right| > 0 ,
\\[0.25in]
\small \left|              
  \begin{array}{ccc}   
    -f_1 & ( b_1-f_2-f_3q )/2 & (\sigma_1-f_3-q )/2 \\[0.05in]
       ( b_1-f_2-f_3q )/2 & b_2 +b_3q & (b_3+\sigma_2+\sigma_3q)/2 \\[0.05in]
    (\sigma_1-f_3-q )/2& (b_3+\sigma_2+\sigma_3q)/2 & \sigma_3 \\  
  \end{array}
\right| > 0 ;
\end{cases}
\eeq
\item[(ii)]  $h>0$, $f_1>0$ and
\beq{cri4}
\begin{cases}
\small\left|                
  \begin{array}{cc}   
    -f_1 & ( b_1-f_2-f_3q )/2  \\[0.05in]
         ( b_1-f_2-f_3q )/2 & b_2 +b_3q  \\
  \end{array}
\right| > 0 ,
\\[0.25in]
\small \left|               
  \begin{array}{ccc}   
    -f_1 & ( b_1-f_2-f_3q )/2 & (\sigma_1-f_3-q )/2 \\[0.05in]
       ( b_1-f_2-f_3q )/2 & b_2 +b_3q & (b_3+\sigma_2+\sigma_3q)/2 \\[0.05in]
    (\sigma_1-f_3-q )/2& (b_3+\sigma_2+\sigma_3q)/2 & \sigma_3 \\  
  \end{array}
\right| < 0 .
\end{cases}
\eeq

\end{itemize}
  \end{corollary}

  To illustrate the results in this section, here we present an example.
 
  \newtheorem{exa1}[theorem]{Example}
\begin{exa1}
{\rm
Consider a linear FBSDEs as follows:
\beq{Ex1}
\begin{cases}
 X(t)=x+\int_0^t[X(s)-Y(s)-2Z(s)]ds+\int_0^t[2Y(s)+Z(s)]dW(s),
 \\
 \\
 Y(t)=-X(T)+\int_t^T[-2X(s)+Z(s)]ds-\int_t^T Z(s)dW(s),  \ \ \ \ \ \ \ \ \ \ \ \  0\leq t \leq T,
\end{cases}
\eeq
In this example, we have $f_1=-2,\ f_2=0,\ f_3=1,\ b_1=1,\ b_2=-1,\ b_3=-2,\ \sigma_1=0,\ \sigma_2=2,\ \sigma_3=1$. Then the coefficients matrix of (\ref{Ex1}) is $\left(                 
  \begin{array}{ccc}   
     2 & 0 & -1\\
    1 & -1 & -2  \\
   0 & 2 & 1  \\
  \end{array}
\right)$ and $h=-1$.
According to Lemma 2.2, it is easy to verify that (\ref{Ex1}) can not match the monotonicity conditions. Note that $h=-1 < \frac{1}{\sigma_3}=1$ , $F(h)=-1 <0$ and $ b_3\sigma_2-b_2\sigma_3=-3<0$, then there exists a constant $\zeta < h$ such that 
$$F(\zeta)>0,$$
which leads to that $F(\cdot)$ has a zero point in $(-\infty,h].$ 
FBSDE~(\ref{Ex1}) is well posed according to the case (i) of Lemma 3.8 (Unified Approach).

To find an equivalent coefficient matrix holding for monotonicity conditions (Lemma 2.2), it is noted that $h=-1 <0$ and $f_1=-2 <0$. Substituting all coefficients into (\ref{cri1}), we can get a feasible interval of $p$:
\bed
\begin{cases}
4p-\frac{9}{4}>0,
\\[0.12in]
-2 p^3+4 p^2+11p-8>0.
\end{cases}
\eed

Let denote $p=1 $ and  (\ref{Ex1}) can be transformed according to Proposition 3.2:

\beq{ex2}
\begin{cases}
 X(t)=x+\int_0^t[X(s)+Y(s)-Z(s)]ds+\int_o^t[2Y(s)+Z(s)]d\tilde{W}(s),
 \\
 \\
 Y(t)=-X(T)+\int_t^T -2X(s)ds-\int_t^TZ(s) d\tilde{W}(s),  \ \ \ \ \ \ \ \ \ \ \ \  0\leq t \leq T,
\end{cases}
\eeq
where $d\tilde{W}(s)=dW(s)-ds$ is a standard Brownian motion under probability measure $\mathbb{Q}$. Here $\mathbb{Q}$ is an equivalent probability measure to $\mathbb{P}$ with 
$$\frac{d\mathbb{Q}\ |\mathcal{F}_t}{d\mathbb{P}\ |\mathcal{F}_t}=\exp{\{-W(t)-\frac{1}{2}t\}}.$$
Note that, in new FBSDE~(\ref{ex2}), we can verify the following relations : 

\bed
\begin{split}
\left(                 
  \begin{array}{ccc}   
    x & y & z \\  
  \end{array}
\right)
&\left(                 
  \begin{array}{ccc}   
     2 & 0 & 0\\
    1 & 1 & -1 \\
   0 & 2 & 1  \\
  \end{array}
\right)
\left(                 
  \begin{array}{c}   
    x \\
    y \\
    z \\  
  \end{array}
\right)=2x^2+xy+y^2+yz+z^2
\\
&=(x+\frac{1}{2}y)^2+x^2+(\frac{2}{3}y+\frac{3}{4}z)^2+\frac{11}{36}y^2+\frac{7}{16}z^2
\\
&\geq x^2+\frac{11}{36}y^2+\frac{7}{16}z^2\geq x^2 +\frac{1}{6}(y^2+z^2)
\end{split}
\eed
Recall that $h=-1$, by taking $\beta_1=1$ and  $\beta_2=\frac{1}{6}$, then the monotonicity conditions (Lemma 2.2) hold according to the case (ii) of Lemma 2.2.
}
\end{exa1}

\section{Linear Transformation Method}
In this section, we consider a linear transformation method for (\ref{LF}) to get another family of FBSDE.
Owing to the non-degeneracy of the transformation matrix, new FBSDE after transformation have the same well-posedness with original FBSDE. Therefore, we could get the well-posedness of original FBSDE if FBSDE after transformation is well posed on $[0,T]$.

Let introduce a non-degenerate $2 \times 2$ matrix $A=\left(  \begin{array}{cc}   
     a_{11} & a_{12} \\
    a_{21} & a_{22}  \\
  \end{array} \right), a_{ij}\in R, i,j=1,2 $.
 Then we consider the following transformation for (\ref{LF}): $$\left(  \begin{array}{c}   
    \tilde{X}(t) \\
    \tilde{Y}(t) \\
  \end{array} \right)=A\left(  \begin{array}{c}   
    X(t) \\
    Y(t) \\
  \end{array} \right)=\left(  \begin{array}{cc}   
     a_{11} & a_{12} \\
    a_{21} & a_{22}  \\
  \end{array} \right)\left(  \begin{array}{c}   
    X(t) \\
    Y(t) \\
  \end{array} \right)=\left(\begin{array}{c}   
   a_{11}X(t)+a_{12}Y(t) \\
    a_{21}X(t)+a_{22}Y(t) \\
  \end{array} \right), \ \  \forall t\in [0,T].$$

Also,
 $$ \begin{cases}
\ X(t)=\displaystyle \frac{a_{22}\tilde{X}(t)-a_{12}\tilde{Y}(t)}{|A|}
\\
\ Y(t)=\displaystyle \frac{-a_{21}\tilde{X}(t)+a_{11}\tilde{Y}(t)}{|A|},
\end{cases}$$
 where $|A|$ represents the determinant of $A$.

 Applying $It\hat{o}\ 's$ formula to $\tilde{X}_t$ and $\tilde{Y}_t$, the original FBSDE ~(\ref{LF}) change into the following form:
 \beq{LT}
  \begin{cases}
 d\tilde{X}(t)=[\tilde{b}_1\tilde{X}(t)+\tilde{b}_2\tilde{Y}(t)+\tilde{b}_3\tilde{Z}(t)]dt+
[\tilde{\sigma}_1\tilde{X}(t)+\tilde{\sigma}_2\tilde{Y}(t)+\tilde{\sigma}_3\tilde{Z}(t)]dW(t),
 \\[0.2in]
 -d\tilde{Y}(t)=[\tilde{f}_1\tilde{X}(t)+\tilde{f}_2\tilde{Y}(t)+\tilde{f}_3\tilde{Z}(t)]dt-\tilde{Z}(t)dW(t),
 \\[0.2in]
 \tilde{X}(0)=\displaystyle \frac{|A|}{a_{22}}x+\frac{a_{12}}{a_{22}}\tilde{Y}(0), Y(T)=\displaystyle \frac{a_{21}+a_{22}h}{a_{11}+a_{12}h}\tilde{X}(T),  \ \ \ \ \ \ \ \ \ \ \ \ 0\leq t \leq T,
\end{cases}
\eeq
where
$$\tilde{b}_1= \frac{(a_{11}b_3-a_{12}f_3)({a_{21}}^2\sigma_2-a_{21}a_{22}\sigma_1)+(a_{22}+a_{21}\sigma_3)[a_{22}(a_{11}b_1-a_{12}f_1)-a_{21}(a_{11}b_2-a_{12} f_2)]}{|A|(a_{22}+a_{21}\sigma_3)},$$

$$\tilde{b}_2=\frac{(a_{11}b_3-a_{12}f_3)(a_{12}a_{21}\sigma_1-a_{11}a_{21}\sigma_2)+(a_{22}+a_{21}\sigma_3)[a_{11}(a_{11}b_2
-a_{12}f_2)-a_{12}(a_{11}b_1-a_{12}f_1)]}{|A|(a_{22}+a_{21}\sigma_3)},$$

$$ \tilde{b}_3=\frac{a_{11}b_3-a_{12}f_3}{a_{22}+a_{21}\sigma_3},$$

$$\tilde{\sigma}_1=\frac{(a_{11}\sigma_3+a_{12})({a_{21}}^2\sigma_2-a_{21}a_{22}\sigma_1)+(a_{22}+a_{21}\sigma_3)(a_{11}a_
{22}\sigma_1-a_{11}a_{21}\sigma_2)}{|A|(a_{22}+a_{21}\sigma_3)},$$

\beq{Co}
\tilde{\sigma}_2=\frac{(a_{11}\sigma_3+a_{12})(a_{12}a_{21}\sigma_1-a_{11}a_{21}\sigma_2)+(a_{22}+a_{21}\sigma_3)({a_{11}
}^2\sigma_2-a_{11}a_{12}\sigma_1)}{|A|(a_{22}+a_{21}\sigma_3)}, 
\eeq

$$\tilde{\sigma}_3=\frac{a_{11}\sigma_3+a_{12}}{a_{22}+a_{21}\sigma_3} ,$$

$$\tilde{f}_1= \frac{(a_{21}b_3-a_{22}f_3)(a_{21}a_{22}\sigma_1-{a_{21}}^2\sigma_2)+(a_{22}+a_{21}\sigma_3)[a_{21}(a_{21}b_2-a_{22}f_2)-a_{22}(a_{21}b_1-a_{22} f_1)]}{|A|(a_{22}+a_{21}\sigma_3)},$$

$$\tilde{f}_2=\frac{(a_{21}b_3-a_{22}f_3)(a_{11}a_{21}\sigma_2-a_{12}a_{21}\sigma_1)+(a_{22}+a_{21}\sigma_3)[a_{12}(a_{21}b_1
-a_{22}f_1)-a_{11}(a_{21}b_2-a_{22}f_2)]}{|A|(a_{22}+a_{21}\sigma_3)},$$

$$ \tilde{f}_3=\frac{a_{22}f_3-a_{21}b_3}{a_{22}+a_{21}\sigma_3},$$

$$ \tilde{Z}(t)=a_{21}\sigma_1X(t)+a_{21}\sigma_2Y(t)+(a_{21}\sigma_3+a_{22})Z(t).$$

\begin{remark}
{\rm
Note that (\ref{LT}) are also the linear FBSDE of which the coefficients become much more complicated after transforming. And also, $(\tilde{X}, \tilde{Y}, \tilde{Z})$ are determined by $(X,Y,Z)$ with respect to transformation matrix $A$.}
\end{remark}

Recall that  
\bed
\begin{split}
\tilde{F}(y)
 &=\frac{-\left|                
  \begin{array}{cc}   
  \tilde{b}_2 & \tilde{b}_3 \\
   \tilde{\sigma}_2 & \tilde{\sigma}_3 \\  
 \end{array}
\right| y^{3}+\left(\tilde{b}_{2}+
\left|                
  \begin{array}{cc}   
    -\tilde{f} _2& -\tilde{f}_3 \\
    \tilde{\sigma}_2 &\tilde{ \sigma}_3\\  
  \end{array}
\right|
-\left|                
 \begin{array}{cc}   
  \tilde{ b}_1 & \tilde{b}_3 \\
  \tilde{ \sigma}_1 & \tilde{\sigma}_3\\  
 \end{array}
\right|\right) y^{2}
+
\left(\tilde{f}_2+\tilde{b}_1+
\left|                
 \begin{array}{cc}   
  -\tilde{f}_1& -\tilde{f}_3 \\
   \tilde{\sigma}_1 & \tilde{\sigma}_3\\  
\end{array}
\right|
\right) y+\tilde{f}_1}{1-\tilde{\sigma}_3 y}.
 \end{split}
 \eed
Then, for the simplicity  of notation, we assume $ \frac{a_{11}}{a_{12}}=m, \frac{a_{12}}{a_{22}}=c$ and $ \frac{a_{21}}{a_{22}}=n$ which leads to $$
A=\left(  \begin{array}{cc}   
     m & 1 \\
    nc & c  \\
  \end{array} \right).$$

  Substituting (\ref{Co}) into $\tilde{F}(y)$, we have 
  \beq{Co2}
 \tilde{\mathcal{L}}(y)=\tilde{F}(y)(1-\tilde{\sigma}_3y)=\Lambda_0 y^3 + \Lambda_1 y^2 + \Lambda_2 y +\tilde{f}_1 ,
  \eeq
    where
 \beq{Co1}
  \small
 \begin{split}
 &\Lambda_0=-\left|                
  \begin{array}{cc}   
  \tilde{b}_2 & \tilde{b}_3 \\
    \tilde{\sigma}_2 & \tilde{\sigma} _3\\  
  \end{array}
\right|
\\
 &=\frac{ -\left|                
 \begin{array}{cc}   
    b_2 & b_3 \\
    \sigma_2 & \sigma_3 \\  
  \end{array}
\right| {a_{11}}^{3}-\left(b_{2}+
\left|                
  \begin{array}{cc}   
    -f_2 & -f_3 \\
    \sigma_2 & \sigma_3 \\  
  \end{array}
\right|
-\left|                
 \begin{array}{cc}   
   b_1 & b_3 \\
   \sigma_1 & \sigma_3 \\  
 \end{array}
\right|\right) {a_{11}}^2 a_{12}
+
\left(f_{2}+b_{1}+
\left|                
  \begin{array}{cc}   
    -f_1 & -f_3 \\
    \sigma_1 & \sigma_3 \\  
  \end{array}
\right|
\right) a_{11}{a_{12}}^2-f_{1}{a_{12}}^3}{-|A| \ (a_{22}+a_{21}\sigma_3)}
\\
  &=\frac{-\left|                
  \begin{array}{cc}   
    b_2 & b_3 \\
    \sigma_2 & \sigma_3 \\  
  \end{array}
   \right| m^{3}-\left(b_{2}+
   \left|                
  \begin{array}{cc}   
    -f_2 & -f_3 \\
    \sigma_2 & \sigma_3 \\  
  \end{array}
\right|
 -\left|                
  \begin{array}{cc}   
    b_1 & b_3 \\
    \sigma_1 & \sigma_3 \\  
  \end{array}
   \right|\right) m^2 
+
 \left(f_{2}+b_{1}+
\left|                
  \begin{array}{cc}   
   -f_1 & -f_3 \\
    \sigma_1 & \sigma_3 \\  
  \end{array}
    \right|
\right) m-f_{1}}{(nc-mc) \ (c+nc\sigma_3)},
\end{split}
\eeq

 \bed
  \small
 \begin{split}
 \Lambda_1 &=\tilde{b}_2+\left|                
  \begin{array}{cc}   
  -\tilde{f}_2 & -\tilde{f}_3 \\
   \tilde{\sigma}_2 & \tilde{\sigma}_3 \\  
 \end{array}
\right| -
\left|                
  \begin{array}{cc}   
   \tilde{b}_1 & \tilde{b}_3 \\
   \tilde{\sigma}_1 & \tilde{\sigma}_3 \\  
  \end{array}
\right|
\\
   &=\frac{-3\left|                
  \begin{array}{cc}   
    b_2 & b_3 \\
    \sigma_2 & \sigma_3 \\  
  \end{array}
    \right| m^2 n - \left(b_{2}+
   \left|                
  \begin{array}{cc}   
    -f_2 & -f_3 \\
   \sigma_2 & \sigma_3 \\  
  \end{array}
    \right|
   -\left|                
  \begin{array}{cc}   
    b_1 & b_3 \\
    \sigma_1 & \sigma_3 \\  
 \end{array}
    \right|\right)( m^2+2mn)}{(n-m)(c+n c \sigma_3)}
   \\
   & \qquad\qquad\qquad\qquad\qquad\qquad\qquad\qquad\qquad\qquad\qquad+  \frac{\left(f_{2}+b_{1}+
\left|                
  \begin{array}{cc}   
    -f_1 & -f_3 \\
    \sigma_1 & \sigma_3 \\  
  \end{array}
\right| \right)(2 m+n)-3f_{1}}{(n-m)(c+n c \sigma_3)},
\end{split}
\eed
 \bed
  \small
 \begin{split}
\Lambda_2 & =\tilde{f}_2+\tilde{b}_1 +\left|                
  \begin{array}{cc}   
   -\tilde{f}_1 & -\tilde{f}_3 \\
    \tilde{\sigma}_1 & \tilde{\sigma}_3 \\  
  \end{array}
\right|
\\
&=\frac{3\left|                
  \begin{array}{cc}   
    b_2 & b_3 \\
    \sigma_2 & \sigma_3 \\  
  \end{array}
\right| m n^2 + \left(b_{2}+
    \left|                
  \begin{array}{cc}   
   -f_2 & -f_3 \\
    \sigma_2 & \sigma_3 \\  
  \end{array}
\right|
   -\left|                
  \begin{array}{cc}   
    b_1 & b_3 \\
    \sigma_1 & \sigma_3 \\  
  \end{array}
\right|\right)( n^2+2mn)}{(n-m)(c+n c \sigma_3)}
    \\
    & \qquad\qquad\qquad\qquad\qquad\qquad\qquad\qquad\qquad\qquad\qquad
- \frac{\left(f_{2}+b_{1}+
\left|                
  \begin{array}{cc}   
    -f_1 & -f_3 \\
    \sigma_1 & \sigma_3 \\  
  \end{array}
\right| \right)(2 n+m)-3f_{1}}{(n-m)(1+n \sigma_3)}.
\end{split}
  \eed 
  
For cases where Theorem 3.9 do not hold, we apply our results to derive some feasible values of $c$ such that the FBSDE after transformation meets the requirement for Theorem 3.9. 

\begin{proposition} Let $\tilde{b}_i, \tilde{f}_i, \tilde{\sigma}_i,\tilde{h}$ be given in (\ref{Co}) and $\tilde{\mathcal{L}}(y)$ take form in (\ref{Co2}). Then $A=\left(  \begin{array}{cc} 
     m & 1 \\
    Nc & c  \\
  \end{array} \right)$ is a linear transformation matrix if one of the following cases hold:

 \begin{itemize}
 \item[(i)] $1-\tilde{\sigma}_3 \tilde{h} > 0, \   \tilde{b}_3\tilde{\sigma}_2-\tilde{b}_2\tilde{\sigma}_3 \leq 0$ and \ $\tilde{\mathcal{L}}(\tilde{h})\cdot \tilde{\sigma}_3 \leq 0$; 
 \item[(ii)] $1-\tilde{\sigma}_3 \tilde{h} < 0, \   \tilde{b}_3\tilde{\sigma}_2-\tilde{b}_2\tilde{\sigma}_3 \geq 0$ and \ $\mathcal{\tilde{L}}(\tilde{h})\cdot \tilde{\sigma}_3 \leq 0$; 
 \item[(iii)] $\tilde{\sigma}_3>0, \ \mathcal{\tilde{L}}(\frac{1}{\tilde{\sigma}_3}) \leq 0$ and $\mathcal{\tilde{L}}(\tilde{h}) \geq 0;$
  \item[(iv)] $\tilde{\sigma}_3<0, \ \mathcal{\tilde{L}}(\frac{1}{\tilde{\sigma}_3}) \geq 0$ and $\mathcal{\tilde{L}}(\tilde{h}) \leq 0;$
 \end{itemize}
\end{proposition}


For cases where the coefficients do not hold for Theorem 3.9, we choose proper values of $c$ according to Proposition 4.2. Owing to the non-degeneracy of transformation matrix $A$, we can get the well-posedness of original FBSDE~(\ref{LF}) by
  \beq{WP}
  \left(
\begin{array}{c}   
     X(t) \\
   Y(t)  \\
  \end{array}
\right)={\left(
\begin{array}{cc}   
     m & 1 \\
   n c & c \\
  \end{array}
\right)}^{-1}\left(
\begin{array}{c}   
        \tilde{X}(t)  \\
   \tilde{Y}(t)   \\
  \end{array}
\right),
\eeq 
$$Z(t)=\frac{\tilde{Z}(t)-nc\sigma_1X(t)-nc\sigma_2Y(t)}{(nc\sigma_3+c)}.$$

By virtue of the Linear Transformation Method, we get a family of FBSDE after transformation which are equivalent to (\ref{LF}). 
Hence we could find out a representative of such family which has a lower coupling structure of FBSDE after transformation (\ref{LT}).
  \begin{proposition}
If $n$ is the zero point of  function $$H(y)=-\left(b_{3} \sigma_{2}-b_{2} \sigma_{3}\right) y^{3}+\left(b_{2}+\sigma_{2} f_{3}-\sigma_{3} f_{2}+\sigma_{1} b_{3}-\sigma_{3} b_{1}\right) y^{2}-\left(f_{2}+b_{1}+f_{3} \sigma_{1}-f_{1} \sigma_{3}\right) y+f_{1},$$  the FBSDE after transformation is partial coupled in the following form:
 \beq{LT2}
  \begin{cases} \displaystyle
 \tilde{X}(t)=\frac{|A|}{a_{22}}x+\frac{a_{12}}{a_{22}}\tilde{Y}(0)+\int_0^t [\tilde{b}_1\tilde{X}(s)+\tilde{b}_2\tilde{Y}(s)+\tilde{b}_3\tilde{Z}(s)]ds+
\int_0^t[\tilde{\sigma}_1\tilde{X}(s)+\tilde{\sigma}_2\tilde{Y}(s)+\tilde{\sigma}_3\tilde{Z}(s)]dW(s),
 \\[0.2in]
 \displaystyle
 \tilde{Y}(t)=\frac{a_{21}+a_{22}h}{a_{11}+a_{12}h}\tilde{X}(T)+\int_t^T[\tilde{f}_2\tilde{Y}(s)+\tilde{f}_3\tilde{Z}(s)]ds-\int_t^T\tilde{Z}(s)dW(s), \ \ \   0\leq t \leq T.
\end{cases}
\eeq

  \end{proposition}
   \begin{proof}
Note that 
\bed 
\begin{split}
 &H(\frac{a_{21}}{a_{22}}) 
= (a_{21}b_3-a_{22}f_3)(a_{21}a_{22}\sigma_1-{a_{21}}^2\sigma_2)+(a_{22}+a_{21}\sigma_3)[a_{21}(a_{21}b_2-a_{22}f_2)-a_{22}(a_{21}b_1-a_{22} f_1)
\end{split}
\eed
Thus $H(n)=0$ leads to 
$$ \tilde{f}_1=0$$
 which completes the proof.
   \end{proof}
        \begin{remark}
   {\rm Besides for $\tilde{f}_1$, one can also get another partial coupled FBSDE after transformation except for $\tilde{\sigma}_3=0$. Here we need to point out that $\tilde{\sigma}_3=0$ contradicts the Proposition 4.2 which has no well-posedness.}
   \end{remark}
   Compared to (\ref{LT}), (\ref{LT2}) is a partial coupled FBSDE which has more applications in the fields of Partial Differential Equation (PDE), stochastic control and other related fields.

   
\section{The applications to Linear Quadratic(LQ) Stochastic Control Problem}
In this section, we illustrate how our results could be applied in optimal LQ control problem. In the stochastic control model, the system is governed by the stochastic differential equation (SDE):
\beq{Csys}
\begin{cases}
\mathrm{d} x(t)=[A(t) x(t)+B(t) u(t)] \mathrm{d} t+[C(t) x(t)+D(t) u(t)] \mathrm{d} W(t), \\[0.1in]
x(0)=x
\end{cases},
\eeq
where $x\in R$ and $A(\cdot), B(\cdot), C(\cdot), D(\cdot) \in L_{\mathcal{F}}^{\infty}([0, T] ; \mathbb{R}^1)$. The control process $u(\cdot) \in \mathcal{U}_{ad}$ is an $\mathcal{F}_t$-adapted process and the cost function to be minimized is defined by

\beq{costf}
J(u(\cdot))=\frac{1}{2} \mathbb{E} \int_{0}^{T}[\langle R(t) x(t), x(t)\rangle+2\langle S(t) u(t), x(t)\rangle+\langle N(t) u(t), u(t)\rangle] \mathrm{d} t+\frac{1}{2} \mathbb{E}[\langle Q x(T), x(T)\rangle],
\eeq
where  $Q\in L_{\mathcal{F}_{T}}^{\infty}(\Omega ; \mathbb{R}^1)$ and $R(\cdot), S(\cdot), N(\cdot) \in L_{\mathcal{F}}^{\infty}([0, T] ; \mathbb{R}^1)$.
\\
~{\bf Problem (LQ).} An admissible control $u^{\star}(\cdot) \in \mathcal{U}_{ad}$ is called optimal if it solves 

\beq{optc}
J(\bar{u}(\cdot))=\inf_{u(\cdot)\in \mathcal{U}_{ad}} J(u(\cdot)).
\eeq
The stochastic maximum principle is one of the fundamental approaches to solve Problem (LQ) and it gives a necessary condition hold by any optimal solution. We apply the maximum principle to Problem (LQ):
\begin{lemma}
Let $\bar{u}(\cdot)$ be an optimal control minimizing the cost function $\mathcal{J}$ over $\mathcal{U}_{ad}$ and let $\bar{x}(\cdot)$ be the corresponding optimal trajectory. Then there exists a pair of processes $(\bar{y}(\cdot), \bar{z}(\cdot)) \in L_{\mathcal{F}}^{\infty}([0, T] ; \mathbb{H}) $ such that the following stochastic Hamiltonian system holds:
\beq{stoH}
\left\{\begin{array}{l}
\mathrm{d} \bar{x}(t)=\left[\left(A-B N^{-1} S\right) \bar{x}(t)-B N^{-1} B \bar{y}(t)-B N^{-1} D \bar{z}(t)\right] \mathrm{d} t \\
\qquad\qquad\qquad\qquad\qquad\qquad +\left[\left(C-D N^{-1} S\right) \bar{x}(t)-D N^{-1} B \bar{y}(t)-D N^{-1} D \bar{z}(t)\right] \mathrm{d} W(t) \\
-\mathrm{d} \bar{y}(t)=\left[\left(R-S N^{-1} S\right) \bar{x}(t)+\left(A-S N^{-1} B\right) \bar{y}(t)+\left(C-S N^{-1} D\right) \bar{z}(t)\right] \mathrm{d} t-\bar{z}(t) \mathrm{d} W(t) \\
\bar{x}(0)=x, \quad \bar{y}(T)=Q \bar{x}(T)
\end{array}\right.
\eeq
And also, the optimal control $\bar{u}(\cdot)$ should take the form:
\beq{OC}
\bar{u}(t)=-S(t)N^{-1}(t) \bar{x}(t)-B(t)N^{-1}(t) \bar{y}(t)-D(t)N^{-1}(t) \bar{z}(t),
\eeq
where $N^{-1}(t)$ denotes the inverse of $N(t)$.
 \end{lemma}
This lemma is a natural consequence of the maximum principle. For more details of the maximum principle and the stochastic Hamiltonian system theory, we refer to the book Yong and Zhou \cite{Yong1999Stochastic} and the reference therein.

Note that the stochastic Hamiltonian system (\ref{stoH}) is of the fully coupled linear FBSDE. Hence we could employ some techniques in the above sections to study the well-posedeness of it.

The coefficients matrix of (\ref{stoH}) is in the following form:
$$\small \left(                 
  \begin{array}{ccc}   
   S N^{-1}S-R & S N^{-1} B-A & S N^{-1} D-C\\
   A-B N^{-1} S & -B N^{-1} B & B N^{-1} D  \\
   C-D N^{-1} S & -D N^{-1} B & -D N^{-1} D \\
  \end{array}
\right).$$
According to the monotonicity conditions (lemma 2.2), we have the following  corollary.
\begin{corollary}
It is easy to get the stochastic Hamiltonian system (\ref{stoH}) would be well-posed, for $\forall t \in [0,T]$ and any bounded process $N(\cdot)$, if one of the following conditions hold:
 \begin{itemize}
 \item[(i)] $N(t)>0$ and \ $S(t) N^{-1}(t)S(t)-R(t)<0$; 
  \item[(ii)] $N(t)<0$ and \ $S(t) N^{-1}(t)S(t)-R(t)>0$.
 \end{itemize}
\end{corollary}
For cases in which the coefficients do not hold for Corollary 5.2, the linear transformation method plays an important role in getting the well-posedness of stochastic Hamiltonian system (\ref{stoH}). Here we present an example to illustrate.

\begin{example}
{\rm
Let $A(\cdot), B(\cdot), C(\cdot), D(\cdot), R(\cdot), S(\cdot), N(\cdot) $ and $Q$ be constants. We consider the following LQ problem:

\beq{sys2}
 \begin{array}{ll} 
\text { Minimize } & J(u(\cdot))=\frac{1}{2} \mathbb{E} \int_{0}^{T}\left[x^{2}(t)+4u(t)x(t)-u^{2}(t)\right] \mathrm{d} t-4\mathbb{E}\left[x^{2}(T)\right] \\[0.15in]
\text { Subject to }  & \left\{\begin{array}{l}
\mathrm{d} x(t)=[x(t)+u(t)] \mathrm{d} t+[x(t)+2u(t)] \mathrm{d} W(t), \\[0.08in]
x(0)=x.
\end{array}\right. 
\end{array} 
\eeq
According to lemma 5.1, LQ problem (J) admits a unique optimal pair $(\bar{u}(\cdot),\bar{x}(\cdot))$ if the following FBSDE are well-posed:
\beq{Ex2}
\begin{cases}
 \bar{x}(t)=x+\int_0^t [3\bar{x}(s)+\bar{y}(s)-2\bar{z}(s)]ds+\int_o^t[5\bar{x}(s)+2\bar{y}(s)+4\bar{z}(s)]dW(s),
 \\
 \\
 \bar{y}(t)=-4\bar{x}(T)+\int_t^T [5\bar{x}(s)+3\bar{y}(s)+5\bar{z}(s)]ds-\int_t^T \bar{z}(s)dW(s),  \ \ \ \ \ \ \ \ \ \ \ \  0\leq t \leq T,
\end{cases}
\eeq
where optimal control $\bar{u}(t)=2\bar{x}(t)+\bar{y}(t)+2\bar{z}(t)$. Note that (\ref{Ex2}) is a fully coupled FBSDE with constant coefficients matrix $\left(                 
  \begin{array}{ccc}   
     -5 & -3 & -5\\
   3 & 1 & 2  \\
   5 & 2 & 4 \\
  \end{array}
\right).$ 

However, we could not get the well-posedness of FBSDE above on $[0,T]$  by any existing methods. Obviously, the monotonicity conditions (Remark 5.2) does not hold for (\ref{Ex2}). In addition, for the Unified Approach (Theorem 3.9), we have 
$$\sigma_3=4>0,\  1-\sigma_3Q=17>0, \ \ \mathcal{L}(y)=-8y^3-23y^2+11y+5,$$
$$\mathcal{L}(Q)= 365 > 0, \ \ \mathcal{L}({\frac{1}{\sigma_3}})=8>0.$$
Thus we can not get the well-posedness of (\ref{Ex2}) according to Theorem 3.9.

By employing the linear transformation method, we need to find a proper transformation matrix $A=\left(  \begin{array}{cc} 
     m & 1 \\
    nc & c  \\
  \end{array} \right)$ such that (\ref{LT}) has a unique solution. To lower the coupling structure of (\ref{LT}), according to Proposition 4.3, we can get 
$$ H(y)=8y^3-23y^2-11y+5.$$
Thus we take $n=-0.658$ one of the zero point of $H(\cdot)$. And, according to Proposition 4.2, to find a triple $(m, c)$ satisfying one of the following inequality systems:
\beq{Ine1}
\large
\begin{cases}
\ 1-\frac{(4m+1)(n c-4c)}{(c+4n c)(m-4)}>0 \quad  \ \  ( < 0),
\\[0.1in]
\ \frac{8 m^{3}-23m^2-11m+5}{(nc-mc) \ (c+4nc)} \leq 0 \quad  \quad ( \geq 0),
\\[0.1in]
\ \mathcal{\tilde{L}}(\frac{nc-4c}{m-4}) \cdot \frac{4m+1}{c+4nc}\leq 0,
\end{cases} 
\text{or} \ \ 
\begin{cases}
\ \frac{4m+1}{c+4nc} > 0 \quad  \quad \ ( <0),
\\[0.1in]
\ \mathcal{\tilde{L}}(\frac{c+4nc}{4m+1})\leq 0 \quad  ( \geq 0),
\\[0.1in]
\ \mathcal{\tilde{L}}(\frac{nc-4c}{m-4})\geq 0 \quad  ( \leq 0),
\end{cases} 
\eeq
where $\mathcal{\tilde{L}}(\cdot)$ takes form in (\ref{Co2}).
For the simplicity of calculation, we take $m=1, c=1$ by (\ref{Ine1}), then we have transformation matrix 
$A=\left(  \begin{array}{cc}   
     1 & 1 \\
    -0.658 & 1  \\
  \end{array} \right). $ This leads to the following new FBSDE after transformation:
 \beq{At} 
 \begin{cases}
 d\tilde{X}(t)=[8.75\tilde{X}(t)-5.11\tilde{Y}(t)+4.29\tilde{Z}(t)]dt+
[-3.87\tilde{X}(t)+1.84\tilde{Y}(t)-3.06\tilde{Z}(t)]dW(t),
 \\
 \\
 d\tilde{Y}(t)=[-0.69\tilde{Y}(t)+2.26\tilde{Z}(t)]dt-\tilde{Z}(t)dW(t),\ \ \tilde{X}_0= 1.658x+\tilde{Y}(0), \ \ \tilde{Y}_0= 1.55\tilde{X}(T), \ \ 
\end{cases}  
\eeq
where $\tilde{Z}(t)=-3.29\bar{x}(t)-1.32\bar{y}(t)-1.63\bar{z}(t)$, $\tilde{h}=1.55$.

 Note that new FBSDE is partial coupled where the coefficient matrix is $\left(                 
  \begin{array}{ccc}   
     0& -0.69 & 2.26\\
   8.75 & -5.11 & 4.29 \\
   -3.87 & 1.84 & -3.06\\
  \end{array}
\right)$. 
Substituting this into (\ref{Co2}), we have 
\beq{L}
\tilde{\mathcal{L}}(y)=-7.76y^3+3.06y^2+18.17y.
\eeq
Note that $$\tilde{\sigma}_3=-3.06,\ \ 1-\tilde{\sigma}_3\tilde{h}=5.74>0,\ \  b_3\sigma_2-b_2\sigma_3=-7.76<0, \ \ \tilde{\mathcal{L}}(\tilde{h})= 6.62>0.$$ 
According to the case (i) of Proposition 4.2, (\ref{At}) exists a unique solution $(\tilde{X}(t),\tilde{Y}(t))$ for $t\in [0,T]$.

Owing to the non-degeneracy of $\left(
\begin{array}{cc}   
     1 & 1 \\
   -0.658 & 1 \\
  \end{array}
\right)$,  we get the unique solution to original FBSDE~(\ref{Ex2}): $$\left(
\begin{array}{c}   
    \bar{x}(t) \\
   \bar{y}(t)  \\
  \end{array}
\right)={\left(
\begin{array}{cc}   
     1 & 1 \\
   -0.658 & 1 \\
  \end{array}
\right)}^{-1}\left(
\begin{array}{c}   
    \tilde{X}(t) \\
   \tilde{Y}(t)   \\
  \end{array}
\right)={\left(
\begin{array}{cc}   
    0.603 &  -0.603 \\
    0.397 & 0.603\\
  \end{array}
\right)}\left(
\begin{array}{c}   
    \tilde{X}(t)  \\
   \tilde{Y}(t)  \\
  \end{array}
\right),$$
$$ \bar{z}(t)=-\frac{\tilde{Z}(t)+3.29\bar{x}(t)+1.32\bar{y}(t)}{1.63},$$
which solves the LQ problem (\ref{sys2}) optimally.}
\end{example}

\section{Conclusion}
In this paper, we investigate two families of coupled FBSDEs. Although the coefficients of these FBSDEs varies a lot, their well-posedness are proved to be equivalent. We firstly prove that, by a series of coefficients matrix, the well-posedness to a family of FBSDEs with different structures are invariant. We also illustrate that such family of FBSDEs are all well posed once we get the well-posedness to one member by any existing methods. 

Secondly, by introducing the linear transformation method, we get another equivalent family of FBSDEs to investigate. More importantly, we could lower the coupling of the original FBSDE without losing the well-posedness which make it possible to solve the fully coupled FBSDE. Owing to the non-degeneracy of transformation matrix, the solution to original FBSDE could be determined by solutions after transformation. 

In addition, we employ our results to study the stochastic LQ control problem with non-standard coefficients.
Besides for stochastic LQ optimal control problems, the linear transformation method could also be applied in other fields, for example, recursive control problem and partial differential equations. 

\section*{Acknowledgments}
Ruyi Liu acknowledges the Natural Science Foundation of China (No. 11971267) and the China Postdoctoral Science Foundation(2021TQ0196). Zhen Wu acknowledges the Natural Science Foundation of China (No. 11831010, 61961160732), the Natural Science Foundation of Shandong Province (No. ZR2019ZD42) and the Taishan Scholars Climbing Program of Shandong (No. TSPD20210302). Detao Zhang acknowledges the Natural Science Foundation of China (No. 11401345)..

\subsection*{Conflict of interest}

The authors declare no potential conflict of interests.

\subsection*{ORCID}

Ruyi Liu: https://orcid.org/0000-0002-7435-857X

Zhen Wu:  https://orcid.org/0000-0003-0758-9463

Detao Zhang: https://orcid.org/0000-0002-8056-9605

\end{document}